\documentclass{amsart}

\usepackage{amssymb}
\usepackage{amsmath}
\usepackage{graphics}
\usepackage{mathrsfs}
\usepackage{color,cite,graphicx}
\usepackage{tikz}
\definecolor{gray}{gray}{0.4}
\usepackage[all]{xy}
\usepackage{enumerate}
 \usepackage{rotating}

\title[Order eight]{Order eight non--symplectic automorphisms on elliptic K3 surfaces}

\author{Dima Al Tabbaa}
\address{Laboratoire de Math\'ematiques et Applications, UMR CNRS 7348,
			Universit\'e de Poitiers, T\'el\'eport 2, Boulevard Marie et Pierre Curie,
			86962 FUTUROSCOPE CHASSENEUIL, France}
\email{Dima.Al.Tabbaa@math.univ-poitiers.fr}

\author{Alessandra Sarti}
\address{Laboratoire de Math\'ematiques et Applications, UMR CNRS 7348,
			Universit\'e de Poitiers, T\'el\'eport 2, Boulevard Marie et Pierre Curie,
			86962 FUTUROSCOPE CHASSENEUIL, France}
\email{sarti@math.univ-poitiers.fr}
\urladdr{http://www-math.sp2mi.univ-poitiers.fr/~sarti/}

\newtheorem{theorem}{Theorem}[section]
\newtheorem{pro}[theorem]{Proposition}
\newtheorem{lemma}[theorem]{Lemma}

\newtheorem{example}[theorem]{Example}
\newtheorem{remark}[theorem]{Remark}
\DeclareMathOperator{\Pic}{Pic}

\DeclareMathOperator{\Fix}{Fix}

\DeclareMathOperator{\tr}{tr}
\DeclareMathOperator{\rk}{rk}
\DeclareMathOperator{\Aut}{Aut}

\newcommand{\IC}{\mathbb{C}}

\newcommand{\IP}{\mathbb{P}}

\newcommand{\lra}{\longrightarrow}

\subjclass[2010]{Primary 14J28; Secondary 14J50, 14J10}

\keywords{non-symplectic automorphisms, K3 surfaces}
\thanks{}

\begin{document}

\begin{abstract}
  In this paper we classify complex $K3$ surfaces with non-symplectic automorphism of order 8  that 
  leaves invariant a smooth elliptic curve. We show that the rank of the Picard group is either 
  $10$, $14$ or $18$ and the fixed locus is the disjoint union of elliptic curves, rational curves and points, whose number does not exceed
  $1$, $2$, respectively $14$. We give examples corresponding to several types of fixed locus in the classification.

\end{abstract}

\maketitle

\section*{Introduction}

The study of automorphisms of K3 surfaces started with the pioneering work of Nikulin \cite{Nikulin1} and several people contributed then to the developpement of the theory, see \cite{zhang} for a short survey. In the paper we investigate purely non-symplectic automorphisms of order $8$, i.e. automorphisms that multiply the non-degenerate holomorphic two form by a primitive 8th root of unity.

The study of non-symplectic automorphisms of prime order was completed by several authors: Nikulin in \cite{nikulinfactor}, Artebani, Sarti and Taki in \cite{AS3,takiauto, ast}. The study of non-symplectic automorphisms of non-prime order turns out to
be more complicated. Indeed, in this situation the ''generic'' case does not imply that the action of the automorphism 
 is trivial on the Picard group \cite[Section 11]{DK}. In the paper \cite{Taki}, Taki completely describes the case 
when  the order of the automorphism is a prime power and the action is trivial on the Picard group. If we consider non-symplectic, non-trivial  automorphisms of  order $2^t$, then by results 
of Nikulin we have $1\leq t\leq 5$, and 
by a recent paper by Taki \cite{Taki32} there is only one $K3$ surface that admits a non-symplectic automorphism of order 32. 
Some further results in this direction are contained in a paper by Sch\"utt  in the case of automorphisms
of a $2$-power order \cite{matthias} and in a paper by Artebani and Sarti  in the case of  order 4 \cite{ASorder4}. Recently in \cite{paper16} the two authors and Taki completed the study for purely non-symplectic automorphisms of order 16. 

This paper mainly deals with purely non-symplectic automorphisms $\sigma$ of order eight on elliptic K3 surfaces under the assumption that thier fourth power $\sigma^4$ is the identity on the Picard lattice. This corresponds to the situation for the generic K3 surface in the moduli space of K3 surfaces with non-symplectic automorphism of order $8$ and fixed action on the second cohomology with integer coefficients, see \cite[Section 10]{DK}. 
The fixed locus $\Fix(\sigma)$ of such an automorphism $\sigma$ is the disjoint union of smooth curves and points. In the paper we give a complete classification in the case that $\Fix(\sigma^4)$ contains an elliptic curve. More precisely let $X$ be a K3 surface, $\omega_X$ a generator of $H^{2,0}(X)$, $\sigma\in \Aut(X)$ such that $\sigma^* \omega_X=\zeta_8\omega_X$, where $\zeta_8$ denotes a primitive 8th root of unity. We denote by $r,l,m$ and $m_1$ the rank of the eigenspaces of $\sigma^*$ in $H^2(X,\IC)$ relative to the eigenvalues $1,-1,i$ and $\zeta_8$ respectively. We also denote by $k$  the number of smooth rational curves fixed by $\sigma$, by $N$ the number of isolated fixed points in $\Fix(\sigma)$. We prove the following result: 
\begin{theorem}\label{intro theorem1}
 Let $\sigma$ be a purely non-symplectic order eight automorphism acting on a $K3$ surface $X$ such that $\sigma^4$ acts identically on $\Pic(X)$. Assume that 
$\Fix(\sigma^4)$ 
contains a curve $C$ of genus $g(C)=1$, then we have the following possibilities for the fixed locus
\begin{itemize}
\item If $\Fix(\sigma)$ contains the elliptic curve $C$, then 
\begin{eqnarray*}
(k,N,\rk\Pic(X))=(0,2,10),(0,4,14).
\end{eqnarray*}
\item If $\Fix(\sigma^2)$ contains the elliptic curve $C$  and there is no elliptic curve  fixed by $\sigma$, then
\begin{eqnarray*}
(k,N,\rk\Pic(X))=(0,2,10),~(0,6,10),~(0,4,14),~(1,10,14).
\end{eqnarray*}
\item If $\sigma^i$ for $i=1,2$ does not fix an elliptic curve, then
\begin{eqnarray*}(k,N,\rk\Pic(X))=(0,2,10),~(0,4,14),~(0,2,14),~(0,6,14),\\
(1,8,14),~(0,2,18),~(0,6,18),(2,14,18).
\end{eqnarray*}
\end{itemize}
\end{theorem}
In Table \ref{table elliptic} we list all the possibilities in detail and we describe the action of $\sigma$ on the fibers of the elliptic fibration. We have in total 16 possibilities and we give examples for all the cases except for some special cases when the action of $\sigma$ is a translation on the invariant elliptic curve or it acts as a rotation on some reducible fibers of the fibration. 

The result of the paper are partially contained in the PhD thesis \cite{dimathesis} of the first author under the supervision of the second author.
 
\section{Basic facts}\label{section fixed locus}
Let $X$ be a $K3$ surface and $\sigma\in\Aut(X)$. We assume that $\sigma^* \omega_X=\zeta_8\omega_X$, where $\zeta_{8}=e^{\frac{2\pi i}{8}}$ is a primitive 8th root of the unity. Such a $\sigma$ is called {\it purely non-symplectic}, for simplicity we just call it  {\it non-symplectic}, always meaning that the action is by a primitive 8th root of unity. 

 We denote by $r_{\sigma^j},l_{\sigma^j},m_{\sigma^j}$ and $m_1$ for $j=1,2,4$ the rank of the eigenspace of $(\sigma^j)^*$ in $H^2(X,\mathbb{C})$ relative to the eigenvalues $1,-1,i$ and $\zeta_8$ 
respectively (clearly $m_{\sigma^4}=0$). For 
simplicity we just write $r,l,m$ for $j=1$. We recall the invariant lattice:
$$
S(\sigma^j)=\{x\in H^2(X,\mathbb{Z})\,|\, (\sigma^j)^*(x)=x\},
$$
and its orthogonal complement: 
$$
T(\sigma^j)=S(\sigma^j)^{\perp}\cap H^2(X,\mathbb{Z}).
$$ 
 
Since the automorphism acts purely non-symplectically,  $X$ is projective, see \cite[Theorem 3.1]{Nikulin1}) so that if we denote 
$\rk S(\sigma^j)=r_{\sigma^j}$, we have that $r_{\sigma^j}>0$ for all $j=1,2,4$ (one can always find an invariant ample class). 
 On the other hand, one can easily show that $S(\sigma^j)\subseteq \Pic(X)$ for $j=1,2,4$  so that the transcendental lattice satisfies $T_X \subseteq T(\sigma^j)$ for $j=1,2,4$. For simplicity we write $T(\sigma):=T(\sigma^1)$. Recall that the action of $\sigma$ on $T_X$ is by primitive 8th roots of the unity, see \cite[Theorem 3.1 (c)]{Nikulin1}.
\begin{remark}\label{rlm}
 It is a straightforward computation to show that the invariants $r_{\sigma^j},l_{\sigma^j},m_{\sigma^j}$ and $m_1$ with $j=1,2,4$ satisfy the following relations:
$$
\begin{array}{l l}
r_{\sigma^2}=r+l; & r_{\sigma^4}=r+l+2m;\\
l_{\sigma^2}=2m; & l_{\sigma^4}=4m_1;\\
2m_{\sigma^2}=4m_1. & 
\end{array}
$$
We remark that the invariants $l_{\sigma^2}$ and $m_{\sigma^2}$ are  even numbers.
\end{remark}
The moduli space of $K3$ surfaces carrying a non-symplectic  automorphism of order 8 with a given action on the $K3$ lattice is known to be the quotient of a complex ball of dimension $m_1-1$, see \cite[$\S 11$]{DK}, 
$$\mathbb{B}=\{[w]\in \mathbb{P}(V): (w,\overline{w})>0\},$$
where $V$ is the $\zeta_8$-eigenspace of $\sigma^*$ in $T(\sigma^{4})\otimes \mathbb{C}$.
This implies that the Picard group of a $K3$ surface, corresponding to the generic point,   
equals $S(\sigma^{4})$ (see \cite[Theorem 11.2]{DK}). 

 By \cite[Theorem 3.1]{Nikulin1} the eigenvalues of $\sigma$ on $T_X$ are primitive $8$th roots of unity so $\textrm{rk}( T_X)=4m_1$. Since $0<\rk( T_X)\leq 21$ we have that $m_1\leq 5$.
\section{The fixed locus}
We denote by $\Fix(\sigma^j)$, $j=1,2,4$ the fixed locus of the automorphism $\sigma^j$ such that
$$
\Fix(\sigma^j)=\{x\in X\,|\,\sigma^j(x)=x\}.
$$
Clearly  $\Fix(\sigma)\subseteq \Fix(\sigma^2)\subseteq \Fix(\sigma^4)$. To describe the fixed locus of order
8 non-symplectic automorphisms we start recalling the following result about non-symplectic involutions (see \cite[Theorem 4.2.2]{nikulinfactor}).
\begin{theorem}\label{nikulin}
 Let $\tau$ be a non-symplectic involution on a $K3$ surface $X$. The fixed locus of $\tau$ is either empty, the disjoint union of two elliptic curves or the disjoint union of a smooth curve of genus 
$g\geq 0$ and $k$ smooth rational curves.

 Moreover, its fixed lattice $S(\tau) \subset \rm{Pic}(X)$ is a 2-elementary lattice with determinant $2^a$ such that:
\begin{itemize}
 \item $S(\tau)\cong U(2)\oplus E_8(2)$ iff the fixed locus of $\tau$ is empty;
\item $S(\tau)\cong U\oplus E_8(2)$ iff $\tau$ fixes two elliptic curves;
\item $2g=22-rk S(\tau)-a$ and $2k=rk S(\tau)-a$ otherwise.
\end{itemize}
\end{theorem}
The action 
of $\sigma$ at a point in $\Fix(\sigma)$ can be locally diagonalized as follows (up to permutation of the coordinates, but this does not play any role in the classification):
$$A_{1,0}=
 \left(
  \begin{array}{ c c }
     \zeta_{8} & 0 \\
     0 & 1
  \end{array} \right),
A_{2,7}=\left(
  \begin{array}{ c c }
     i & 0 \\
     0 & \zeta^{7}_{8}
  \end{array} \right),
 A_{3,6}=
 \left(
  \begin{array}{ c c }
     \zeta^{3}_{8} & 0 \\
     0 & -i
  \end{array} \right),
A_{4,5}=\left(
 \begin{array}{ c c }
     -1 & 0 \\
    0 & \zeta^{5}_{8}
  \end{array} \right).
$$
In the first case the point belongs to a smooth fixed curve, while in the other cases it is an isolated fixed point. We say that an isolated point $x \in \Fix(\sigma)$ is of {\it type $(t,s)$} if the
local action at $x$ is given by $A_{t,s}$. We denote by $n_{t,s}$ the number of isolated
 fixed points by $\sigma$ with matrix $A_{t,s}$. We further denote by $N_{\sigma^j}$, $k_{\sigma^j}$, $j=1,2,4$ the number of isolated points and smooth rational curves in $\Fix(\sigma^j)$. We observe that $N_{\sigma^4}=0$ since $\sigma^4$ only fixes curves (or is empty) as explained
in Theorem \ref{nikulin}. For simplicity we write $N:=N_{\sigma}$ and $k:=k_{\sigma}$. The fixed locus of an automorphism $\sigma$ 
is then
$$
\Fix(\sigma)=C\cup R_1\cup\ldots\cup R_k\cup\{p_1,\ldots,p_N\}
$$
where $C$ is a smooth curve of genus $g\geq 0$, $R_i$ are smooth disjoint rational curves and $p_i$ are isolated 
points. We will see that the fixed locus can never contain two elliptic curves or be empty.
\begin{pro}\label{proposition1}
Let $ \sigma$ be a purely non-symplectic automorphism of order 8 acting on a $K3$ surface $X$. Then $\Fix(\sigma)$ is never empty nor it can be the union of two smooth elliptic curves. It the disjoint union of smooth curves and
$N\geq 2$ isolated points. Moreover, the following relations hold:

$$ n_{2,7} + n_{3,6} = 2+4\alpha,\qquad n_{4,5}+n_{2,7}-n_{3,6} =  2+2\alpha, \qquad N=2+r-l-2\alpha .$$
Here we denote  $\alpha =\sum\limits_{K\subset \Fix(\sigma)}(1-g(K))$.
\end{pro}
\begin{proof}
Let $\sigma$ be a purely non-symplectic automorphism of order 8, then $\sigma^\ast(\omega_{X})=\zeta_{8} \omega_{X}$ where $\zeta_{8}=e^\frac{2\pi i}{8}$. The Lefschetz number of $\sigma$ is :
\begin{eqnarray*}
L(\sigma)& = &\sum_{j=0}^{2}(-1)^j(\tr(\sigma^\ast|H^j(X,\mathcal{O}_{X}))=1+\zeta_{8}^{7},  
\end{eqnarray*}
By \cite[Theorem 4.6]{Atiyah Singer} we have also
$$
L(\sigma)=\sum_{t,s} \frac{n_{t,s}}{\det(I-\sigma^*|T_x)}+\frac{1+\zeta_{8}}{(1-\zeta_{8})^2}\sum_{C\subset \textrm{Fix}(\sigma)} (1-g(C))
$$
where $T_x$ denotes the tangent space at an isolated fixed point $x$. Comparing the two expressions for $L(\sigma)$ and since 
$L(\sigma)\neq 0$ the fixed locus is never empty nor it can be the union of two elliptic curves. Then using the expression for the local action
of $\sigma$ at a fixed point we get the equations:

\begin{eqnarray}\label{equation1}
\left\{
 \begin{array}{r c l}
n_{2,7} + n_{3,6}& = &2+4\alpha .\\
n_{4,5}+n_{2,7}-n_{3,6}& = & 2+2\alpha.
\end{array}
\right.
\end{eqnarray}
Since $\alpha\geq 0$ we observe that $\sigma$ fixes at least two isolated points.
We consider now the topological Lefschetz fixed point formula for $\sigma$. Recall that we write $r=r_\sigma$ and $l=l_\sigma$.
We have
$$
\begin{array}{r l}

 N+\sum_{K\subset\Fix(\sigma)}(2-2g(K))&=\chi(\textrm{Fix}(\sigma))=\sum_{h=0}^{4}(-1)^h \textrm{tr}(\sigma^*|H^h(X,\mathbb{R}))\\&=2+\textrm{tr}(\sigma^*|H^2(X,\mathbb{R})).
\end{array}
$$
This gives $N+2\alpha=\chi(\textrm{Fix}(\sigma))=r-l+2$, so that $r-l=N+2\alpha-2$ and we have the third relation in the statement.
\end{proof}
\begin{remark}\label{Remark3}
  The isolated fixed points, by a non-symplectic order $eight$ automorphism $\sigma$  of type $(2,7)$ and $(3,6)$ are also  isolated fixed points in $\Fix(\sigma^2)$. The points of type $(4,5)$ in $\Fix(\sigma)$ are contained in a smooth fixed curve by $\sigma^2$.
In fact the action of $\sigma^2$ at such points is given by the matrix 
 $\left(
  \begin{array}{ c c }
     1 & 0 \\
     0 & \zeta_{8}^{2}
  \end{array} \right)$ which implies that these points belong to a smooth curve in $\rm{Fix(\sigma^2)}$.
 \end{remark}
From now on we  denote by $n_t=n_{t,s}$ the number of  isolated fixed points by $\sigma$  of type $(t,s)$, where $t=2,3,4$ and $t+s=9$. Recall the usefull lemma.\\

\begin{lemma}\label{Lemma 4}\cite[Lemma 8.1]{jimmy},\cite{ASorder4}
 Let $\tau=\sum_i R_i$ be a tree of smooth rational curves on a $K3$ surface $X$ such that each $R_i$ is invariant under the action of a purely non-symplectic automorphism $\sigma$ of order $q$. Then, the 
points of intersection of the rational curves $R_i$ are fixed by $\sigma$ and the action at one fixed point determines the action on the whole tree.
\end{lemma}
\begin{remark} \label{localact}
 In the case of an automorphism of order 8, with the assumptions of Lemma \ref{Lemma 4}, the local actions at the intersection points of the curves $R_i$ appear in the following order (we give only the 
exponents of $\zeta_8$ in the matrix of the local action):
$$\ldots,(0,1),(7,2),(6,3),(5,4),(4,5),(3,6),(2,7),(1,0),\ldots$$
Assuming that $\tau=R$ consists only of one rational curve, which is not pointwise fixed and do not intersect a fixed curve of higher genus,  one  get immediately that $\sigma$ has either one fixed point of type $(2,7)$ and another one of type $(3,6)$, two fixed  points of type $(4,5)$ or one fixed point of type $(4,5)$ and one of type $(3,6)$.  
\end{remark}
\begin{pro}\label{Le2}
Let $\sigma$ be a non-symplectic automorphism of order 8 on a $K3$ surface $X$. Assume that $\Pic(X)=S(\sigma^4)$ and $C\subset \Fix(\sigma)$ with $g(C)=1$. Then the following relations hold:
$$
\begin{array}{c c l}
 
N_{\sigma^2}&=&2k_{\sigma^2} +4,\\
4k_{\sigma^2}&=&r_{\sigma^2}-l_{\sigma^2}-2=2(10-l_{\sigma^2}-m_{\sigma^2}).
\end{array}
$$
\end{pro}
\begin{proof}
 Observe that the fixed locus of $\sigma^2$ is the disjoint union of $k_{\sigma^2}$ smooth rational curves, a smooth elliptic curve and $N_{\sigma^2}$ isolated points. Moreover we have that an isolated fixed point by 
$\sigma^2$ is given by the local action  $\left(
  \begin{array}{ c c }
     -i & 0 \\
     0 & -1
  \end{array} \right)$. We obtain the relations in the statement by applying holomorphic and topological Lefschetz's formulas.
\end{proof}
\section{The classification}
 \begin{lemma}\label{8basis}
 If a K3 surface $X$ carries a $\sigma$-invariant elliptic fibration, such that 
 $\sigma^4$ fixes an irreducible smooth fiber $C$ of this fibration, then $\sigma$ acts with order 8 on the basis of the fibration and fixes two points on it.
 \end{lemma}
 \begin{proof}
 Let $\pi_{C}:X \longrightarrow \mathbb{P}^1$ be a $\sigma-$invariant elliptic fibration having $C$ as a smooth fiber and such that $C$ is fixed by the involution $\sigma^4$. Observe that $\sigma^2$ 
(respectively~$\sigma^4$) is not the identity on the basis of $\pi_C$, since otherwise it would act as the identity on the tangent space at a point of $C$, 
contradicting the fact that $\sigma^2$ (respectively $\sigma^4$) is purely non-symplectic. Hence $\sigma$ acts as an order eight automorphism on $\mathbb{P}^1$ and has two fixed points on it corresponding to $C$ and
 another fiber $C'.$
 \end{proof}
\begin{remark}\label{elliptic}
 Let $C$ be an elliptic curve fixed by the involution $\sigma^4$. Then $C$ is also invariant by $\sigma^i$ for $i=1,2$. And $\sigma$ either fixes $C$ or it acts on $C$ as a translation with no-fixed points
or $\sigma$ preserves $C$ and fixes isolated points on it. In this last case we have two possibilities:
\begin{itemize}
 \item[a)] The automorphism $\sigma$ acts on $C$ as an automorphism of order four with two isolated fixed points, that must be of type $(2,7)$ and/or $(3,6)$. In fact $\sigma$ can not have a fixed point of type $(4,5)$ on $C$ otherwise this point would be contained
on a fixed curve for $\sigma^2$, that would be also fixed by $\sigma^4$, but $\sigma^4$ already fixes $C$ and fixed curves do not intersect. 

\item[b)] The automorphism $\sigma$ acts on $C$ as an involution with four isolated fixed points. These must be clearly of type $(4,5)$. 
\end{itemize}
\end{remark}
\begin{theorem}\label{thmelliptic}
 Let $\sigma$ be a purely non-symplectic automorphism of order 8 on a $K3$ surface $X$, such that the involution $\sigma^4$ fixes a smooth elliptic curve. Then the number $\# C$ of  fixed smooth elliptic curves by 
$\sigma^4$ is at most $2$. Moreover, the action of $\sigma$ is described in the  Table \ref{table elliptic}.
\end{theorem}
\begin{proof}
Observe first that $\rk\Pic(X) \in \{ 18,14,10,6,2\}$ since $\varphi(8) | \rk T_X$ by \cite{Nikulin1}. The involution $\sigma^4$ fixes an elliptic curve by Theorem \ref{nikulin} 
(see also  \cite[Figure 1]{ast})
so that we get $\textrm{rk~Pic}(X)\in \{10,14,18\}$. On the other hand, by Theorem \ref{nikulin} again we have that $\sigma^4$ fixes at most two elliptic curves, moreover we know that for the number $\# C$ of smooth fixed 
elliptic curves by $\sigma^4$ and the number $k_{\sigma^4}$ of rational fixed curves holds:
\begin{equation}\label{cases}
 (\rk \Pic (X),\# C, k_{\sigma^4})=(10,2,0),(14,1,4),(18,1,8).
\end{equation}
Let $C$ denote a $\sigma^4$-fixed smooth elliptic curve. Since $\sigma$ preserves $C$ there is a $\sigma$-invariant elliptic fibration 
$$
\pi_C:X\lra\IP^1
$$
with generic fiber $C$. Observe that by Lemma \ref{8basis} the automorphism $\sigma$ has order eight on the basis $\IP^1$ and it fixes two points  corresponding to the fiber $C$ and a fiber $C'$. The fiber $C'$ can be smooth elliptic and so $k_{\sigma}=0$ or it can be reducible and then it contains all rational curves fixed by $\sigma^4$. In 
the case $(\rk \Pic (X),\# C, k_{\sigma^4})=(14,1,4)$, the curve $C'$ is a reducible fiber containing four smooth rational curves fixed by $\sigma^4$. Since the non-symplectic involution $\sigma^4$ does not fix isolated points and the fixed curves by $\sigma^4$
 do not intersect, one can easily see that the reducible fiber $C'$ is either of Kodaira type $IV^*$ or $I_8$. By the same argument we get that $C'$ is of type $I_{16}$ in the case $(\rk \Pic (X),\# C, k_{\sigma^4})=(18,1,8)$ (see also \cite[Theorem 3.1, Theorem 8.4]{ASorder4}).

 \textbf{\underline{The case $\rk \Pic(X)=10$}}. Since $k_{\sigma^4}=0$ we have also $k=0$, thus $n_2+n_3=2$ by Proposition \ref{proposition1}. Observe that $\sigma$ acts as an automorphism of order four on one of the two
$\sigma^4$-fixed elliptic curves $C,C'$, in fact $\sigma^2$ has order four and can not fix two elliptic curves, see \cite[Proposition 1]{ASorder4}. On the other hand  either $\sigma$  acts on the other $\sigma^4-$fixed elliptic curve 
as a translation (of order 2 or 4, recall that the fixed locus of $\sigma^4$ consists of two smooth eliptic curves) or it is the identity, in this case $n_4=0$ so that $(n_2,n_3,n_4)=(2,0,0)$ or it acts as an involution when  $(n_2,n_3,n_4)=(0,2,4)$ (see Proposition \ref{proposition1} and Remark \ref{elliptic}). Since $\rk\Pic(X)=10$ we have that $\rk T_X=12$ so that $m_1=3$. This gives $m_{\sigma^2}=6$ by Remark \ref{rlm}. Since $k_{\sigma^2}=0$ using Proposition 
\ref{Le2} we get that $10-l_{\sigma^2}-m_{\sigma^2}=0$, so we have $l_{\sigma^2}=2m=4$, $m=2$ and $r_{\sigma^2}=6$ by Proposition \ref{Le2} again. By Proposition \ref{proposition1}
and Remark \ref{rlm} one can find easily that $(r,l,m)=(3,3,2)$, respectively $(5,1,2)$. This gives the {\bf first four cases} in Table \ref{table elliptic}.

\textbf{\underline{The case $\rk \Pic(X)=14$}}. As we have already seen, the fiber $C'$ is either of type $IV^{*}$ or $I_8$. We study here these two possibilities separately.
\begin{itemize}
 \item \underline{ $C'$ of type $IV^{*}$:} Since the fiber $IV^{*}$ is preserved by $\sigma^i~;~i=1,2,4$, the automorphism $\sigma$ either preserves each component of $IV^{*}$ or it exchanges the two branches
 of $IV^{*}$. In these two cases  $\sigma^2$ preserves each component of $C'$ and fixes the central component of multiplicity 3 since it has at least three fixed points by $\sigma^2$. Then we have that 
$k \in \{ 0,1\}$ by Remark \ref{localact}. 
If $k=1$ then $\sigma$ preserves each component of $C'$, hence it fixes six points on $C'$ three of each type $(2,7)$ and $(3,6)$ by Remark \ref{localact}, thus 
$(n_2,n_3,n_4) \geq (3,3,0)$ (we write $(x_1,x_2,x_3)\geq (y_1,y_2,y_3)$ if and only if $x_i\geq y_i$ for all $i=1,2,3.$). By Proposition \ref{proposition1} we get $(n_2,n_3,n_4)=(3,3,4)$, this implies that $\sigma$ acts as an involution on $C$ and fixes four points on it of type $(4,5)$.
This is {\bf case 11} in Table \ref{table elliptic}.\\
If $k=0$  the automorphism $\sigma$ exchanges the two branches of $C'$. So the fiber $C'$ contains two fixed points of type $(4,5)$ on the central 
component of multiplicity 3, which is then fixed by $\sigma^2$, and one point of type $(2,7)$ and another one of type $(3,6)$, hence we get $(n_2,n_3,n_4) \geq (1,1,2)$.
 By Proposition \ref{proposition1}
we have that $(n_2,n_3,n_4)=(1,1,2)$ which means that  $\sigma$ acts on $C$ either as the identity or as a translation of order two or four (since $\sigma^4$ fixes $C$). These are the {\bf cases 5,6,7} in the Table \ref{table elliptic}.\\
Finally, since $k_{\sigma^2}=1$ and by doing the same computation as in the previous case one gets immediately that $(r,l,m)=(10,0,2)$ and $(6,4,2)$  respectively for $k=1$ and $k=0$.
\item \underline{ $C'$ of type $I_8$:} First observe  that all components of $C'$ are preserved by $\sigma^4$ since $\textrm{Pic}(X)= S(\sigma^4)$ and a component which is not fixed intersects two 
fixed ones. The automorphism $\sigma$ either preserves each component of $I_8$ or 
acts on it as a reflection (i.e. $\sigma$ preserves two components on $I_8$ and exchanges the remaining 6 components 
two by two) or it acts as a rotation, of order two or four. Applying Remark \ref{localact} 
one gets that in the first case $\sigma$ fixes one component of $C'$ (i.e. $k=1$) and 2 isolated points of 
each type $(2,7), (3,6)$ 
and $(4,5)$, so we get $(n_2,n_3,n_4) \geq (2,2,2)$. On the other hand, we have that $ (n_2,n_3,n_4) = (4,2,2)$ 
by Proposition \ref{proposition1},  which implies that $\sigma$ acts on $C$ as an 
automorphism of order four and  fixes two points on it of type $(2,7)$ (see Remark \ref{elliptic}). This is the {\bf case 12}
in the Table \ref{table elliptic}. 

If $\sigma$ acts on $C'$ as a reflection, then $k=0$ and $\sigma$ fixes four points of 
type $(4,5)$ on the two invariants components of $C'$. Observe that these two components are fixed by $\sigma^2$. 
Hence we get $(n_2,n_3,n_4) \geq (0,0,4)$. On the other hand, by Proposition \ref{proposition1}   we have that $(n_2,n_3,n_4) = (0,2,4)$, thus 
$\sigma$  acts as an automorphism of order four on $C$ and fixes two points of type $(3,6)$ (see Remark \ref{elliptic}). This is the {\bf case 10} in the Table \ref{table elliptic}.

Using the relation $4k_{\sigma^2}=r_{\sigma_2}-l_{\sigma^2}-2=2(10-l_{\sigma^2}-m_{\sigma^2})$ of Proposition \ref{Le2} since $k_{\sigma^2}=2$ and
$m_{\sigma^2}=4$ we get in 
these two cases that $(r_{\sigma_2},l_{\sigma^2})=(12,2)$. In fact in these two cases  $\sigma^2$ preserves each component of $I_8$ and fixes two components in $C'$.  We get then easily that  $(r,l,m)$ is equal respectively to  $(10,2,1)$ and $(8,4,1)$  by Proposition \ref{proposition1}.

 Finally, if $\sigma$ acts as a rotation of order two or of order four on $C'$, then in these both cases $k=0$ and $\sigma$ does not fix any point on $C'$.  By Proposition \ref{proposition1} we have
that $(n_2,n_3,n_4) = (2,0,0)$ and so $\sigma$ also acts as an automorphism of order four on $C$ and fixes two points of type $(2,7)$ (see Remark \ref{elliptic}). Moreover we compute as above the invariants $r, l, m$ in the case that  $k_{\sigma^2}=2$ respectively $0$ 
which correspond to the action of the automorphism on $I_8$ as a rotation  of order two or of order four respectively. We get easily that $(r,l,m)$ equals $(6,6,1)$ respectively $(4,4,3)$. These are the {\bf cases 8 and 9} in the Table \ref{table elliptic}.
\end{itemize}
 \textbf{\underline{The case $\rk \Pic(X)=18$}}. In this case the fiber $C'$ is of type $I_{16}$ as remarked at the beginning of the proof. As we have seen previously, there are four possibilities for the action of $\sigma$ on $I_{16}$ 
that are: either $\sigma$ preserves each component of $I_{16}$ or it acts as a reflection (here $\sigma$ preserves two components of $I_{16}$ and exchanges the remaining $14$ components two by two) or it acts 
as a rotation of order two, respectively four. IThese are the {\bf cases 16, 15, 13 and 14} in the Table \ref{table elliptic}. If $\sigma$ preserves each component of $C'$, then by applying Remark \ref{localact} we get that $\sigma$ fixes two components of $C'$ (i.e. $k=2$) and $12$ isolated points
four of each type $(2,7), (3,6)$ and $(4,5)$. On the other hand, if $\sigma$ acts as a reflection, then $k=0$ and $\sigma$ fixes four isolated points of type $(4,5)$
on the two preserved components of $C'$. Finally, $\sigma$ does not fix any point on $C'$ if it acts as a rotation.\\
Observe that in these three cases $\sigma$ acts as an automorphism of order four on the smooth $\sigma^4-$fixed fiber $C$ by Proposition \ref{proposition1}. Using the same argument as before we get the values of 
the invariants $r, l, m$ that appeared in  the last four cases of Table \ref{table elliptic}.
\end{proof}
\section{Examples}\label{section examples}

In this section we give examples corresponding to several cases in the classification of the non-symplectic automorphisms of order eight on elliptic $K3$ surfaces. 
\begin{example}\label{example1}
\rm{Consider the elliptic fibration $\pi_C : X\longrightarrow \mathbb{P}^1$ with a Weierstrass equation:
$$y^2= x^3 + a(t) x + b(t).$$
where~~$ a(t) =at^8 + b$ ~ and ~$b(t) = ct^8 + d$ with $a,b,c,d \in \mathbb{C}$. The fibration $\pi_C$ admits the order eight automorphism: $$\sigma(x,y,t) = (x,y,\zeta_{8} t).$$
The fibers preserved by $\sigma$ are over $ 0$ and  $\infty$ and the action of $\sigma$ at infinity is: 
$$(x/t^4 , y/t^6 , 1/t)\longmapsto (-x/t^4 , iy/t^6 , \zeta_{8} ^7/t)$$
The discriminant polynomial of $\pi_C$ is:
$$\Delta(t) :=4a(t)^3 + 27 b(t)^2 = h_1t^{24} + h_2t^{16} + h_3t^8 +O(t^4), $$
where $$h_1=4a^3 , ~~h_2=12 a^2b+ 27c^2,~~h_3=12ab^2+54cd.$$ 
Observe that $\Delta(t)$ has 24 simple zeros for a generic choice of the coefficients. By studying the zeros of $\Delta(t)$ and looking in the classification of singular fibers of elliptic
 fibrations on surfaces (see e.g. \cite[section 3]{miranda}) one obtains the following:
for a generic choice of the coefficients of $a(t)$ and $b(t)$ the fibration has $24$ fibers of type $\rm{I}_1$ over the zeros of $\Delta(t)$, both fibers over $t=0$ and $t=\infty$ are smooth elliptic
curves, moreover the automorphism  
 $\sigma $ fixes pointwisely the fiber over $0$ and  
acts as an \textit{order 4 automorphism} on the fiber over $\infty$, hence it has two fixed points on this fiber. We show below that generically
the rank of $\Pic(X)$ is 10. This gives an example for the \textbf{first case in Table \ref{table elliptic}}. On the other hand, if $h_1=0$ so that $a=0$ the fibration acquires a fiber of type $\rm{IV}^\ast$ at $\infty$ by a generic choice 
of the parameters. The  fiber  over 0 is a smooth elliptic curve fixed pointwisely by $\sigma$. This gives an example for the \textbf{case 5 in Table \ref{table elliptic}}, this follows by Theorem \ref{elliptic} and by the fact that this case is the only one for which $\sigma$ contains an elliptic curve
in the fixed locus and the fibration has a fiber of type $IV^*$.  

 By using standard transformations on the parameters in the Weierstrass form we get that the number of moduli is 2. In fact both the polynomials $a(t) , b(t)$ depend on 2 parameters, but we can apply the transformation $(x,y)\mapsto (\lambda^2x,\lambda^3y)~;~\lambda \in \mathbb{C}^{\ast}$, 
to cancel one of the 4 parameters. Moreover the automorphisms of $\mathbb{P}^1$ commuting with $t \mapsto \zeta_8t$ are of the form $t \mapsto \mu t~;~\mu \in \mathbb{C}^{\ast}$, so we can cancel a second parameter. This shows that the family depends on 2 parameters, so that generically $\rk T_X=12$ (recall that $\rk T_X=4m_1$ and $m_1-1$ equals to the number of moduli) and $\rk\Pic(X)=10$. 
 }
\end{example}
\begin{example}\label{example(elliptic-4)}

\rm{ As in the previous example \ref{example1}, consider again the elliptic fibration $\pi_C : X\longrightarrow \mathbb{P}^1$ in Weierstrass form given by :
$$y^2= x^3 + a(t) x + b(t).$$
where $ a(t) =at^8 + b$  and $b(t) = ct^8 + d$ with $a,b,c,d \in \mathbb{C}$. This elliptic fibration carries the non-symplectic automorphism $\sigma$  of order eight: $$(x,y,t) \mapsto (x,-y,\zeta_{8} t).$$
The fibers preserved by $\sigma$ are over $ 0 , \infty$ and the action at infinity is 
$$(x/t^4 , y/t^6 , 1/t)\longmapsto (-x/t^4 , -iy/t^6 , \zeta_{8} ^7/t).$$
The discriminant polynomial of $\pi_C$ is:
$$\Delta(t) :=4a(t)^3 + 27 b(t)^2 = h_1t^{24} + h_2t^{16} + O(t^8), $$
where
$$h_1=4a^3~\mbox {and}~h_2=12 a^2b+ 27c^2.$$ 

 We have seen in example \ref{example1} that for a generic choice of the coefficients of $a(t)$ and $b(t)$ the fibration has $24$ fibers of type $\rm{I}_1$ over the zeros of $\Delta(t)$ (see \cite[section 3]{miranda}),  moreover 
$\sigma $ acts as an $involution$ on the fiber over $0$ and it acts as an \textit{order 4 automorphism} on the fiber over $\infty$ (both fibers are smooth). So we have an example 
for the \textbf{ case 4 in Table \ref{table elliptic}}.
If $h_1=0$, that implies $a=0$, the fibration acquires a fiber of type $\rm{IV}^\ast$ at $\infty$ by a generic choice of the parameters. This is the \textbf{ case 11 in Table \ref{table elliptic}} (this is the only case where the fibartion has a fiber of type $IV^*$ and the action on the smooth fiber
is an involution).} 
\end{example}
\begin{example}\label{example (elliptic-2)_2}
\rm{ Consider  the elliptic fibration $\pi_C :X\longrightarrow \mathbb{P}^1$ in Weierstrass form  given by $$y^2 = x^3 + a(t)x + b(t) ,$$
where $a(t) = at^8 + b$ and $b(t)=ct^4 + dt^{12}~;~a,b,c,d \in \mathbb{C}$. Observe that it carries the order eight automorphism $$\sigma : (x,y,t) \longmapsto (-x,iy ,\zeta_{8}^{7}t).$$
For generic choice of the coefficients the fiber over $t=0$ is smooth.
The automorphism $\sigma^4$ is an involution fixing the smooth elliptic curve over $t=0$. On the other hand, $\sigma$ acts on the elliptic curve over
$t=0$ as an \textit{order 4 automorphism} with $2$ isolated fixed points. Moreover it acts as the identity on the fiber over $t=\infty$, in fact the action is 
$$(x/t^4 , y/t^6 , 1/t )\longmapsto (x/t^4, y/t^6 , \zeta_{8}/t). $$
The discriminant of $\pi_C$ is :
$$\Delta(t) = h_1t^{24} + h_2t^{16} +h_3t^8 + 4b^3 ,$$
where
$$h_1=4a^3 + 27d^2 ~, ~h_2=12a^2b + 54cd ~,~ h_3=12ab^2 + 27c^2.$$
Observe that $\Delta(t)$ has 24 simple zeros for a generic choice of the coefficients. By studying the zeros of $\Delta(t)$ and looking in the classification of singular fibers of elliptic fibrations on surfaces (e.g \cite[section 3]{miranda}) one obtains the following cases: for the generic choice of the coefficients of $a(t)$ and $b(t)$ the fibration has 24 fibers of type $I_1$. We get again an example for {\bf case 4 in Table \ref{table elliptic}}. If $h_1=0$ the fibration has a fiber of type $\rm{I}_8$ at infinity.
If $h_1=h_2=0$ we get a fiber $\rm{I}_{16}$. By \cite[$\S$3]{kondoell} a holomorphic two form is given by $\omega_X = (dt \wedge dx)/2y$ and so the action of $\sigma$ on it is 
$\sigma^{\ast}(\omega_X)=\frac{-\zeta^{7}_{8}}{i}\omega_X=\zeta_8 \omega_X$. We  determine now exactly the type of the local action of $\sigma$ at the two fixed points on the elliptic 
curve $C$. We look at the elliptic fibration locally around the fiber over $t=0$. The equation in $\mathbb{P}^2 \times \mathbb{C}$ is given by:
$$
F(x,y,z,t):= zy^2-(x^3+(at^8+b)z^2x+(ct^4+dt^{12})z^3)=0. 
$$ 
Where $(x:y:z)$ are the homogeneous coordinates of $\mathbb{P}^2$ . The fiber at $t=0$ has equation 
$$
f:=\{ zy^2-x^3-b z^2x=0\}
$$ 

and for $b\in\IC$ generic is smooth since the partial derivatives $\partial f/\partial x , \partial f/\partial y, \partial f/\partial z$ are not simultaneously zero. The two fixed points by $\sigma$ are $p:=(0:1:0)$ and $p':=(0:0:1)$. The point 
$p':=(0:0:1)$ is contained in the chart $z=1$ and it belongs to the open set $F_x:=\partial F/ \partial x \neq 0$, in fact $F_x(p'):=\frac{\partial F(x,y,1,0)}{\partial x} \neq 0$. The 
 one-form for the elliptic curve in this open subset is:
$$dy/F_x(p') = dy/(-3x^2-b). $$
The action of $\sigma$ here is a multiplication by $i$ : ($\sigma^{\ast}(dy/F_x(p'))=i(dy/F_x(p'))$), so that the action on the holomorphic 2-form $dt \wedge (dy/(-3x^2(at^8+b))$ is
the multiplication by $\zeta_8$ as expected. In particular the local action at $p'$ is of type (7,2). Similarly we can do the computation 
on the open subset in the chart $y=1$ which contains the fixed point $p$, and we can find again the same action (7,2). Observe that since $\sigma$ acts as the identity on the fiber over $t=\infty$, it preserves the curves of the singular fibers (of type $I_8$ or $I_{16}$). This gives  an example for \textbf{ the  cases 12 and 16 in Table \ref{table elliptic}}. 

 On other hand, the fibration $\pi_C$ admits also the automorphism $\tau (x,y,t)=(-x,-iy,\zeta^{3}_{8} t)$. This automorphism acts also by multiplication by $\zeta_8$ on the 
holomorphic 2-form $\omega_X$, thus $\tau$
 is not  a power of $\sigma$ (i.e it is a new automorphism). Moreover the square of $\tau$ preserves each components 
of the fiber at $t=\infty$  in fact the action at infinity is:
$$(x/t^4,y/t^6,1/t)\longmapsto (x/t^4,-y/t^6,\zeta^{5}_{8}/t). $$
By a similar computation as above one sees 
that the local action at the fixed points on the fiber $C$ is of type (3,6), so we have an example for \textbf{the cases 10 and 15 in Table \ref{table elliptic}} respectively.}
\end{example}
\begin{example}\label{translation}{\it (Translation)}.\\
\rm{We give here an example for the \textbf{cases 2 , 8 and 13 in Table \ref{table elliptic}}. Observe that in these three cases the 
non-symplectic automorphism
 of order 8 on $X$ acts as a translation of order two on the $\sigma$-invariant fiber $C'$. 

 It is well known that an elliptic $K3$ surface with a 2-torsion section can be written with equation:
$$
y^2=x(x^2+a(t)x+b(t),
$$
where $a(t)$ and $b(t)$ are polynomials of degree 4 and 8 respectively. Or equivalently one can write the equation in Weierstrass form: 
$$
y^2=x^3+A(t)x+B(t),
$$
where $$A(t)=9b(t)-3a(t)^2~ \mbox{and}~ B(t)=3a(t)^2-9a(t)b(t).$$
Now the map: 
\begin{equation}\label{4}
\tag{4}
\tau:(x,y,t) \mapsto (y^2/x^2-a(t)-x,(y/x).\tau(x),t),
\end{equation}
with  $\tau(x):=(y^2/x^2)-a(t)-x$ is an automorphism on $X$ that acts as a translation of order two on the generic fiber of $\pi$.

 Consider now the non-symplectic automorphism of order eight on $X$:
 $$\sigma : (x,y,t) \mapsto (-x,iy,\zeta^{7}_{8}t).$$
 This automorphism preserves the jacobian elliptic fibration $\pi : X \longrightarrow \mathbb{P}^1$ defined as follows:
 $$y^2=x(x^2+a(t)x+b(t)),$$
 where $a(t)=\alpha t^4, b(t)=\beta t^8+\gamma~; ~\alpha , \beta , \gamma \in \mathbb{C}$. Or equivalently the fibration can be written as:
 $$y^2=x^3+A(t)x+B(t),$$
 such that $A(t)=(9\beta-3\alpha^2)t^8 +9\gamma$ and $B(t)=(2\alpha^3-9\alpha \beta)t^{12}-9\alpha\gamma t^4$. The discriminant of $\pi$ is:
 $$\Delta(t)=K(\beta t^8 + \gamma)^2[(\alpha^2-4\beta)t^8-4\gamma];~ K\in \mathbb{C} \rm{~is~a~ constant}.$$
 Consider now the translation $\tau$ : 
 $$\tau: (x,y,t) \mapsto (y^2/x^2-\alpha t^4-x, (y/x).\tau(x),t).$$
 As we have seen, $\tau$ is an automorphism of $X$ and it acts as a translation of order two on the generic fiber of $\pi$. 
 Moreover, one can get easily that $\tau \circ \sigma = \sigma \circ \tau$, thus the $K3$ surface $X$ has the order eight non-symplectic automorphism $\sigma' :=\sigma \circ \tau $. Observe that the automorphisms $\sigma$ and $\sigma'$ act with order eight on $\mathbb{P}^1$, they preserve the two fibers over $t=0$ and $t=\infty$ and act as an automorphism of order four on the smooth fiber over $t=0$ given by  
 $f:= zy^2-x^3-9\gamma xz^2$, for $\gamma\in\IC$ generic. Moreover, $\sigma$ acts as the identity on the fiber over $t=\infty$, while $\sigma'$ acts on it as an order two translation (observe that the action of $\sigma$ at infinity is $(x/t^4,y/t^6,1/t) \mapsto (x/t^4,y/t^6,\zeta _8 1/t)$ and $\sigma'^2=(\sigma \circ \tau)^2=id$).
 
  Studying the zeros of the discriminant $\Delta(t)$, and looking in the classification of singular fibers of elliptic fibrations on surfaces (e.g \cite[$\S$3]{miranda}) we get the following:
\begin{itemize}
 \item For generic $\alpha,\beta,\gamma$ the fibration $\pi$ has 8 fibers of type $I_2$ and 8 fibers of type $I_1$ over the zeros of $\Delta(t)$, $\sigma'$ acts as an order
 four automorphism on the fiber over 0 and it acts as an order two translation on the fiber at $\infty$ (both fibers are smooth). This gives an example for the \textbf{case 2 in Table \ref{table elliptic}}.
  \item If $\alpha^2-4\beta=0$ with $(\beta \neq 0)$, then $\Delta(t)=K(\beta t^8+\gamma)^2(-4\beta)~;~K\in\mathbb{C}^{\ast}.$ So that the fibration acquires a fiber of type $I_8$ over $\infty$ 
(in fact $\Delta(t)$ has a zero of order 8 over $t=\infty$ and $A(t),B(t)$ are non zero). The automorphism $\sigma'$ acts as a rotation of order two on the fiber $I_8$ (see \cite{miranda}). This corresponds to the \textbf{case 8 in Table \ref{table elliptic}}.
  \item If $\beta=0,~ (\alpha \neq 0)$,  the discriminant $\Delta(t)=K\gamma^2(\alpha^2 t^2-4\gamma)$ vanishes at $t=\infty$ with order 16, and $A(t) , B(t)$ 
are nonzero at $\infty$. Thus we get a fiber of type $I_{16}$, on which $\sigma$ acts as a rotation of order two. We are in the \textbf{case 13 of Table \ref{table elliptic}}. 
  \end{itemize}
}
\end{example}
\begin{example}{\it Quadruple Quartics}.\\
\rm{Take the fourfold cover of $\mathbb{P}^2$:
$$
t^4=x_0(l_3(x_1,x_2)+x_0^2l_1(x_1,x_2))
$$
where $l_3(x_1,x_2)$ is homogeneous of degree three and $l_1(x_1,x_2)$ is homogeneous of degree 1. This is invariant for the action of the order $8$ non symplectic automorphism:
$$
(t, x_0, x_1, x_2)\mapsto (\zeta_8 t, -x_0, x_1, x_2)
$$
it fixes the inverse image of the curve $\{x_0=0\}$ which is rational and $4$  points on the curve $C:\{l_3(x_1,x_2)+x_0^2l_1(x_1,x_2)=0\}$ which is in fact elliptic. 
This gives another example for the \textbf{case 11 of Table \ref{table elliptic}}. }
\end{example}
\section{The Table of the $\sigma$-invariant elliptic fibrations\label{table elliptic}}
We give the table for
 the classification of the non-symplectic automorphisms of order 8 on an elliptic K3 surface. The cases for which we have an example 
are denoted with $\surd$. We  list also   
in this table the invariants  $r, l, m$ and $m_1$ of $\sigma$ which denote the rank of the eigenspaces of $(\sigma)^*$ in $H^2(X,\mathbb{C})$ relative to the eigenvalues $1,-1,i$ and $\zeta_8$ 
respectively. Recall that by $\# C$ we denote the number of $\sigma$-invariant smooth elliptic curves. 

\newpage

{\small
\begin{center}
\begin{rotate}{-90}
\begin{tabular*}{1.50\textwidth}{@{\extracolsep{\fill}}|c|c c c|c|c|c|cc|c|c|c|}
\hline
   Example&$r$&$l$&$m$&$k_{\sigma^2}$&$\# C$& $\rk \Pic(X)$& $k_{\sigma^4}$&$N$& $(n_2,n_3,n_4)$& $k$& action of $\sigma$ on fixed elliptic curves and singular fibers\\
\hline
1. $\surd$&3 &3 &2 &0 &2 &10 &0 &2 &(2,0,0) &0 &( identity, order four)\\
       &  &  &  &  &  &   &  &  &        &  &                        \\
2. $\surd$&3 &3 &2 &0 &  &  &  &2 &(2,0,0) &0 &(translation of order two, order four)\\
 &  &  &  &  &  &   &  &  &        &  &                        \\
3. $-$&3 &3 &2 &0 &  &  &  &2 &(2,0,0) &0 &(translation of order four, order four)\\
 &  &  &  &  &  &   &  &  &        &  &                        \\
4. $\surd$&5 &1 &2 &0 &  &  & &6  &(0,2,4) &0 &(involution, order four)\\
 &  &  &  &  &  &   &  &  &        &  &                        \\
\hline
\hline
5. $\surd$&6 &4 &2 &1 &1 &14 &4 &4 &(1,1,2 )&0 &( identity, reflection of $IV^*$ )\\
 &  &  &  &  &  &   &  &  &        &  &                        \\
6. $-$&6 &4 &2 &1 &  &   &  &4 &(1,1,2)&0&(translation of order two, reflection of $IV^*$ )\\
 &  &  &  &  &  &   &  &  &        &  &                        \\
7. $-$&6 &4 &2 &1 &  &   &  &4 &(1,1,2)&0&(translation of order four, reflection of $IV^*$ )\\
 &  &  &  &  &  &   &  &  &        &  &                        \\
8. $\surd$&6 &6 &1 &2 &  &   &  &2&(2,0,0)&0&(order four, rotation of order $2$ on $I_8$ )\\
 &  &  &  &  &  &   &  &  &        &  &                        \\
9. $-$&4 &4 &3 &0 &  &   &  &2&(2,0,0)&0&(order four, rotation of order $4$ on $I_8$)\\
 &  &  &  &  &  &   &  &  &        &  &                        \\
10. $\surd$&8 &4 &1 &2 &  &   & &6&(0,2,4)&0&(order four, reflection on $I_8$)\\
 &  &  &  &  &  &   &  &  &        &  &                        \\
11. $\surd$&10&0 &2 &1 &  &   &&10&(3,3,4)&1&(involution, preserves each curve of $IV^*$)\\
 &  &  &  &  &  &   &  &  &        &  &                        \\
12. $\surd$&10&2 &1 &2 &  &   & &8&(4,2,2)&1&(order four, preserves each curve of $I_8$)\\
 &  &  &  &  &  &   &  &  &        &  &                        \\
\hline
\hline

13. $\surd$&9 &9 &0 &4 &1 &18 &8 &2 &(2,0,0)&0&(order four, rotation of order $2$ on $I_{16}$ )\\
 &  &  &  &  &  &   &  &  &        &  &                        \\
14. $-$&5 &5 &4 &0 &  &   &&2&(2,0,0)&0&(order four, rotation of order $4$ on $I_{16}$ )\\
 &  &  &  &  &  &   &  &  &        &  &                        \\
15. $\surd$&11&7 &0 &4 &  &   & &6&(0,2,4)&0&(order four, reflection on $I_{16}$ )\\
 &  &  &  &  &  &   &  &  &        &  &                        \\
16. $\surd$&17&1 &0 &4 &  &   & &14&(6,4,4)&2&(order four, preserves each curve of $I_{16}$ )\\
\hline
\end{tabular*}
\end{rotate}
\end{center} 
}
\newpage
\bibliographystyle{amsplain}
\bibliography{Biblio}
\end{document}